\documentclass[11pt]{smfart}
\usepackage{color}
\usepackage{amssymb,verbatim}
\usepackage{hyperref}
\usepackage{amsthm,array,amssymb,amscd,amsfonts,amssymb,latexsym, url}
\usepackage{amsmath}
\usepackage[all]{xy}
\usepackage[french]{babel}

\newcounter{spec}
{\end{list}}

\renewcommand{\P}{{\mathbf P}}
\newcommand{\A}{{\mathbf A}}

\newcommand{\Z}{{\mathbb Z}}

\renewcommand{\L}{{\mathbb L}}

\renewcommand{\lim}{\varprojlim}

\numberwithin{equation}{section}

\newfont{\gothic}{eufb10}

\renewcommand{\qed}{{\hfill$\square$}}

\newtheorem{theo}{Th\'{e}or\`{e}me}[section]
\newtheorem{prop}[theo]{Proposition}

\newtheorem{lem}[theo]{Lemme}
\newtheorem{cor}[theo]{Corollaire}
\theoremstyle{definition}
\newtheorem{defi}[theo]{D\'efinition}
\theoremstyle{remark}
\newtheorem{rema}[theo]{Remarque}

\newcommand{\bthe}{\begin{theo}}
\newcommand{\ble}{\begin{lem}}
\newcommand{\bpr}{\begin{prop}}
\newcommand{\bco}{\begin{cor}}
\newcommand{\bde}{\begin{defi}}
\newcommand{\ethe}{\end{theo}}
\newcommand{\ele}{\end{lem}}
\newcommand{\epr}{\end{prop}}
\newcommand{\eco}{\end{cor}}
\newcommand{\ede}{\end{defi}}

\DeclareFontFamily{U}{wncy}{}
\DeclareFontShape{U}{wncy}{m}{n}{%
<5>wncyr5%
<6>wncyr6%
<7>wncyr7%
<8>wncyr8%
<9>wncyr9%
<10>wncyr10%
<11>wncyr10%
<12>wncyr6%
<14>wncyr7%
<17>wncyr8%
<20>wncyr10%
<25>wncyr10}{}
\DeclareMathAlphabet{\cyr}{U}{wncy}{m}{n}

\begin{document}

  \title[Droites sur les hypersurfaces cubiques] {Droites sur les hypersurfaces cubiques}

\author{J.-L. Colliot-Th\'el\`ene}
\address{  Universit\'e Paris-Sud, CNRS,  Paris-Saclay\\Math\'ematiques, B\^atiment 425\\91405 Orsay Cedex\\France}
\email{jlct@math.u-psud.fr}

\date{28 mars 2017}
\maketitle

 \begin{abstract}
 On montre que sur toute hypersurface cubique  complexe de dimension
 au moins 2, le groupe de Chow des cycles  de dimension 1 est engendr\'e par les droites.
 Le cas lisse est un th\'eor\`eme connu. La d\'emonstration
 ici donn\'ee repose sur un r\'esultat sur les surfaces g\'eom\'etriquement rationnelles 
 sur un corps quelconque (1983), obtenu via la K-th\'eorie alg\'ebrique.
  \end{abstract}
  
  \begin{altabstract}
  Over any complex cubic hypersurface of dimension at least 2,
  the Chow group of 1-dimensional cycles is spanned by the lines
  lying on the hypersurface. The  smooth case has already been given
  several other proofs.
   \end{altabstract}

\section*{Introduction}

Soit $X$ une vari\'et\'e sur un corps quelconque. On note $CH_{i}(X)$ le groupe de Chow des cycles de dimension $i$ sur $X$ modulo l'\'equivalence rationnelle.
Le th\'eor\`eme suivant est connu.

\begin{theo}\label{lisse}
Soient $k$ un corps alg\'ebriquement clos de caract\'eristique z\'ero et
$X \subset \P^n_{k}$, avec $n \geq 3$, une hypersurface cubique lisse.
Le groupe de Chow $CH_{1}(X)$  des 1-cycles de $X$ modulo \'equivalence rationnelle
est engendr\'e par les droites contenues dans $X$.
$\Box$
\end{theo}

Pour $n=3$, c'est un r\'esultat  classique.
Pour $n= 6$, c'est \'etabli par Paranjape \cite[\S 4]{paranjape}.
Celui-ci utilise l'existence d'un $\P^2$ contenu dans $X \subset \P^6$
pour fibrer $X \subset \P^6$ en quadriques de dimension 2 au-dessus de $\P^3$.
Paranjape \'ecrit qu'une m\'ethode analogue vaut pour tout $n \geq 6$.
Pour tout $n \geq  4$, le th\'eor\`eme est \'etabli par M. Shen \cite[Thm. 1.1]{shen}
par une m\'ethode diff\'erente de celle de Paranjape.
Dans la r\'ecente pr\'epublication, M. Shen \cite[Thm. 4.1 ]{shen2} \'etend ce  r\'esultat
en un th\'eor\`eme  pour toute hypersurface cubique lisse de dimension au moins 3
sur un corps de base non n\'ecessairement alg\'ebriquement clos, lorsque l'hypersurface cubique contient
une droite d\'efinie sur ce corps.
Pour $n \geq 5$, le th\'eor\`eme est  aussi un cas particulier d'un r\'esultat de Tian et Zong
sur les   intersections compl\`etes de Fano \cite[Thm. 6.1]{TZ} de dimension $m$
et de multidegr\'e $(d_{1},\cdots,d_{c})$ avec $d_{1}+ \cdots d_{c} \leq m-1$
(r\'esultat obtenu par encore une autre m\'ethode).

Comme le note d\'ej\`a  Paranjape \cite{paranjape}, 
 pour $n \geq 6$, l'\'enonc\'e implique que le groupe de Chow $CH_{1}(X)$
 est \'egal \`a $\Z$. En effet
le sch\'ema de Fano des droites de $X$ est alors une vari\'et\'e de Fano (lisse, projective, faisceau anticanonique ample),
 et un  th\'eor\`eme bien connu
de Campana 
 et de Koll\'ar-Miyaoka-Mori  
 dit que les vari\'et\'es de Fano sont
rationnellement connexes (par cha\^{\i}nes).

Dans cette note, je donne une nouvelle d\'emonstration du th\'eor\`eme \ref{lisse}, qui vaut
pour les hypersurfaces cubiques quelconques. C'est le th\'eor\`eme \ref{sing}.
Il y a deux ingr\'edients.
Le premier ingr\'edient est un r\'esultat sur les surfaces projectives  lisses g\'eom\'etriquement rationnelles  sur les corps de dimension cohomologique~1  (\cite{CT},
 th\'eor\`eme \ref{K2} ci-dessous),
 dont la d\'emonstration utilise la K-th\'eorie alg\'ebrique (th\'eor\`eme de Merkur'ev et Suslin).  
Le second ingr\'edient est classique : c'est la  classification des types de surfaces cubiques  singuli\`eres 
 sur un corps alg\'ebriquement clos. La d\'emonstration proc\`ede par sections hyperplanes et r\'ecurrence sur la dimension.
M\^eme pour une hypersurface cubique lisse donn\'ee, elle  impose de consid\'erer toutes les hypersurfaces cubiques  de dimension un de moins obtenues par section hyperplane,  et celles-ci peuvent \^etre singuli\`eres.

Nous n'utiliserons que les propri\'et\'es les plus simples  des groupes de Chow des vari\'et\'es,
telles qu'on les trouve dans le chapitre 1 du livre   \cite{fulton}, en particulier la suite de localisation
\cite[Prop. 1.8]{fulton} et le comportement dans une fibration en droites affines \cite[Prop. 1.9]{fulton}.

\'Etant donn\'ee une vari\'et\'e $X$ projective sur un corps $K$, la R-\'equivalence sur l'ensemble
$X(K)$ des points $K$-rationnels de $X$ est la relation d'\'equivalence engendr\'ee par
la relation \'el\'ementaire suivante : deux $K$-points $A$ et $B$ sont \'el\'ementairement li\'es
s'il existe un $K$-morphisme $f : \P^1_{K} \to X$ tel que $A$ et $B$ soient dans $f(\P^1(K))\subset X(K)$.
Si deux $K$-points $A$ et $B$ sont R-\'equivalents, alors $A-B=0 \in CH_{0}(X)$.

\section{Groupe de Chow des z\'ero-cycles d'une hypersurface cubique sur un corps de fonctions d'une variable}

Le th\'eor\`eme suivant est \ 
une cons\'equence imm\'ediate de \cite[Prop. 4]{CT},
puisque le groupe de cohomologie galoisienne $H^1(K,S)$
pour un $K$-tore $S$ sur un corps $K$ de dimension cohomologique 1
est nul.
\begin{theo} \cite[Theorem A (iv)]{CT}  \label{K2}
Soit $K$ un corps de caract\'eristique z\'ero et de dimension cohomologique 1.
Soit $X$ une $K$-surface projective, lisse, g\'eom\'etriquement rationnelle.
Le noyau de l'application degr\'e $deg_{K} : CH_{0}(X) \to \Z$
est nul. Si $X$ poss\`ede un point rationnel, par exemple si $K$ est
un corps $C_{1}$, alors l'application degr\'e $$deg_{K} : CH_{0}(X) \to \Z$$
est un isomorphisme. $\Box$
\end{theo}

Ce th\'eor\`eme s'applique en particulier aux surfaces cubiques lisses.
\'Etudions maintenant le cas des surfaces cubiques quelconques.

\begin{prop}  \label{surfaces}
 Soit $K$ un corps de caract\'eristique z\'ero et de dimension cohomologique 1.
Soit $X \subset \P^3_{K}$ une surface cubique.
Supposons $X(K) \neq \emptyset$, ce qui est le cas si $K$ est $C_{1}$,
par exemple si $K$ est un corps de fonctions d'une variable sur un corps
alg\'ebriquement clos.
Alors l'application degr\'e
$$deg_{K}: CH_{0}(X) \to \Z$$ est un isomorphisme.
\end{prop}
  \begin{proof}
Comme toute surface cubique lisse sur un corps alg\'ebri\-que\-ment clos
est rationnelle, le cas o\`u $X$ est lisse est un cas particulier du th\'eor\`eme \ref{K2}.

Supposons $X$ singuli\`ere. Si   $X \subset \P^3_{K}$ est un c\^one,
tout point ferm\'e de $X$ est rationnellement \'equivalent \`a un multiple
d'un point $K$-rationnel du sommet du c\^one (cet argument vaut sur un corps
quelconque). 

Si $X$ n'est pas un c\^one, mais n'est pas g\'eom\'etriquement 
int\`egre, alors c'est l'union d'un plan  $P$ et d'une quadrique $Q$  g\'eom\'etriquement int\`egre,
leur intersection est une conique $C$ dans $\P^2_{K}$. Toute telle conique poss\`ede un
point $K$-rationnel, puisque $cd(K)\leq 1$, et  $deg_{K} : CH_{0}(C) \to \Z$ est un isomorphisme.
Fixons $m \in C(K)$.
Tout point  ferm\'e du plan $P$ est rationnellement \'equivalent \`a un multiple de $m$.
Si la quadrique $Q$ est un c\^one de sommet $q \in Q(K)$, tout point ferm\'e
de $Q$ est rationnellement \'equivalent \`a un multiple de $q$, et $m$
est rationnellement \'equivalent \`a $q$. Si la quadrique $Q$ est lisse, alors
elle est $K$-rationnelle car elle poss\`ede un $K$-point, et $deg_{K} : CH_{0}(Q) \to \Z$ est un isomorphisme (en fait $Q(K)/R=\{*\}$).
On conclut que  $deg_{K} : CH_{0}(X) \to \Z$ est un isomorphisme.

Supposons d\'esormais que la surface cubique $X \subset \P^3_{K}$ n'est pas un c\^one et est g\'eom\'etriquement  int\`egre. Elle est alors g\'eom\'etriquement rationnelle.
 Les diverses singularit\'es possibles ont \'et\'e analys\'ees depuis longtemps
(Schl\"affli, Cayley, B. Segre, Bruce--Wall \cite{brucewall}, Demazure, Coray--Tsfasman   \cite{coraytsfasman}).

Si les points singuliers ne sont pas isol\'es, alors la surface cubique $X$ contient une droite double $D \subset X$,  qui est d\'efinie sur $K$. 
Tout $K$-point de $X$ hors de $D$ est situ\'e sur une droite d\'efinie sur $K$
rencontrant $D$, \`a savoir la droite r\'esiduelle de l'intersection avec $X$  du plan d\'efini
par $D$ et le $K$-point.
On a donc $X(K)/R=\{*\}$ et $deg_{K} : CH_{0}(X) \to \Z$ est un isomorphisme.

Supposons  d\'esormais de plus que  les points singuliers de $X$ sont  isol\'es. 

Si $X$ poss\`ede un point singulier $K$-rationnel,
alors $X(K)/R=\{*\}$ \cite[Lemme 1.3]{madore2008}, sous la simple hypoth\`ese que toute conique
sur $K$ poss\`ede un point rationnel. On a donc alors $X(L)/R= \{*\}$ pour toute extension finie 
de corps $L/K$. Ainsi $deg_{K} : CH_{0}(X) \to \Z$ est un isomorphisme.

Supposons dor\'enavant de plus que l'on a $X_{sing}(K)=\emptyset$. Soit  $f: Y \to X$ une r\'esolution des singularit\'es. 
Un argument simple (lemme de Nishimura) montre que l'application induite $Y(K) \to X(K)$
contient les $K$-points lisses de $X$ dans son image. Donc  $Y(K) \to X(K)$ est surjectif.
Par hypoth\`ese, on a $X(K)\neq \emptyset$.   Soient $P$ et $Q$ deux $K$-points de $X$.
Soient $M$, resp. $N$, dans
$Y(K)$ d'image $P$, resp. $Q$, dans $X(K)$. La $K$-surface $Y$ est projective, lisse, g\'eom\'etriquement rationnelle. Le th\'eor\`eme \ref{K2} assure    $M-N =0 \in CH_{0}(Y)$.
Le morphisme propre $f$  induit $f_{*} : CH_{0}(Y) \to CH_{0}(X)$.
On a donc $P-Q=0 \in CH_{0}(X)$.  Si $R$  est un point ferm\'e de $X$, de corps r\'esiduel $L=K(R)$, suivant que 
  $X_{L}$ poss\`ede un $L$-point singulier ou non, l'un des deux arguments ci-dessus
  garantit $R-M_{L} = 0 \in CH_{0}(X_{L})$, et donc  $deg_{K} : CH_{0}(X) \to \Z$ est un isomorphisme.
 \end{proof}

\begin{theo}\label{generique}
  Soit $K$ un corps de caract\'eristique z\'ero et de dimension cohomologique 1. Soient $n \geq 3$ et
 $X \subset \P^n_{K} $, $n \geq 3$ une hypersurface cubique. Si l'on $X(K) \neq \emptyset$,
 par exemple si $K$ est un corps $C_{1}$,
 l'application degr\'e
$deg_{K}: CH_{0}(X) \to \Z$ est un isomorphisme.
\end{theo}

\begin{proof} 
Soit $O$ un point $K$-rationnel et $P$ un point ferm\'e de $X$, de corps r\'esiduel
$L=K(P)$. Sur $X_{L} \subset \P^n_{L} $, on dispose d'un point $L$-rationnel $p$
d\'efini par $P$ et du $L$-point  $q= O_{L}$.
On choisit un espace lin\'eaire $H \subset \P^n_{L}$
de dimension 3
qui contient $p$ et $q$. Soit $Y:= X_{L} \cap  H$. Si $Y=H$,
alors $p$ et $q$ sont  $R$-\'equivalents sur $X_{L}$, donc $p-q = 0 \in CH_{0}(Y)$.
Si $Y \subset H$ est une surface cubique, le th\'eor\`eme pr\'ec\'edent
assure aussi $p-q = 0 \in CH_{0}(Y)$ et donc $p-q=0 \in CH_{0}(X_{L})$.
Ainsi $P- [L:K]O = 0 \in CH_{0}(X).$
\end{proof}
\begin{rema} Pour  tout corps $K$ qui est $C_{1}$, et tout $n \geq 5$, un argument \'el\'ementaire
  \cite[Prop. 1.4]{madore2008}   montre  que l'on a $X(K)/R=\{*\}$ 
pour toute hypersurface cubique (lisse ou non), d'o\`u il r\'esulte imm\'ediatement que $deg_{K}: CH_{0}(X) \to \Z$ est un isomorphisme  \cite[Cor. 1.6]{madore2008}. C'est une question ouverte si sur un tel corps $K$,
et d\'ej\`a sur un corps $K$ de fonctions d'une variable sur le corps des complexes,
on a $X(K)/R=1$ pour toute hypersurface
cubique  lisse $X \subset \P^n_{K}$ pour $n=3,4$.
\end{rema}

\section{Groupe de Chow des 1-cycles d'une hypersurface cubique sur un corps alg\'ebriquement clos}

\begin{theo} \label{sing}  
Soit $k$ un corps alg\'ebriquement clos de caract\'eristique z\'ero.
Soit $X \subset \P^n_{k}$, avec $n \geq 3$ une hypersurface cubique.
Le groupe de Chow $CH_{1}(X)$ est  engendr\'e par
les droites contenues dans $X$.
\end{theo}
\begin{proof}  On va \'etablir cet \'enonc\'e par r\'ecurrence sur $n \geq 3$.
On commence par \'etablir le cas $n=3$ par une discussion cas par cas.

Dans un plan $\P^2$ tout 1-cycle est rationnellement \'equivalent \`a  un multiple d'une droite.
Pour une quadrique $Q \subset \P^3$ non singuli\`ere, le groupe de Picard de $Q$
est engendr\'e par les deux classes de g\'en\'eratrices. 
Si $Y \subset \P^3$ de coordonn\'ees $(x,y,z,t)$ est un c\^one d\'efini par une \'equation $f(x,y,z)=0$,
et de sommet $p$ de coordonn\'ees $(0,0,0,1)$,  
 $CH_{1}(Y)=CH_{1}(Y\setminus p)$ est engendr\'e par les g\'en\'eratrices du c\^one.
Ceci \'etablit le r\'esultat dans le cas o\`u la surface cubique n'est pas int\`egre, et aussi
dans le cas o\`u c'est un c\^one.

Supposons donc $X$ int\`egre et non conique. Si les singularit\'es de $X$ ne sont pas
isol\'ees, alors $X$ poss\`ede une droite double.  
On peut alors \cite[\S 2, Case E]{brucewall} 
trouver des coordonn\'ees homog\`enes
$(x,y,z,t)$ de $\P^3$ telles que la surface soit donn\'ee soit par l'\'equation
$$x^2z+y^2t=0$$
soit par l'\'equation
$$x^2z+xyt+y^3=0.$$
Dans le premier cas, le compl\'ementaire des deux droites $x=y=0$ et $x=t=0$, d\'ecoup\'ees par $x=0$,
est isomorphe au plan affine $\A^2$ de coordonn\'ees $(y,t)$.
Dans le second cas, le compl\'ementaire de la droite $x=y=0$ d\'ecoup\'ee par $x=0$
est isomorphe au plan affine $\A^2$ de coordonn\'ees $(y,t)$. Comme on a $CH_{1}(\A^2)=0$,
ceci \'etablit que $CH_{1}(X)$ est engendr\'e par des droites de $X$.

Sinon, $X$ est normale, et si $f: X' \to X$ est sa d\'esingularisation minimale,
alors $X'$ est une surface de del Pezzo g\'en\'eralis\'ee de degr\'e 3,
et les ``droites'' de $X'$
sont les transform\'ees propres des vraies droites de $X$.
Voir l\`a-dessus \cite[Exemple 0.5]{coraytsfasman}.
La projection $CH_{1}(X') \to CH_{1}(X)$ est clairement surjective,
et le groupe  $CH_{1}(X')=Pic(X')$ est engendr\'e par les ``droites'' de $X'$ (courbes
$D$ lisses de genre z\'ero avec  $(D.D)=-1$ 
et
les ``racines irr\'eductibles'' (courbes lisses de genre z\'ero avec $(D.D)=-2$) 
qui sont des courbes contract\'ees par $f$
sur les points singuliers de $X$.
Donc $ CH_{1}(X)$ est engendr\'e par les vraies droites de $X \subset \P^3$.

Soit $n\geq  4$. Supposons le cas $n-1$ \'etabli. Soit $X \subset \P^{n}$ une hypersurface cubique.
On trouve dans $X$ une droite $D$ (il en existe sur toute surface cubique sur $k$ alg\'ebriquement clos) et  on choisit $Q\simeq \P^{n-2} \subset \P^{n}$
un espace lin\'eaire de dimension $n-2$   
qui ne rencontre pas $D$ et qui n'est pas contenu dans $X$.  
On consid\`ere  le pinceau des  espaces lin\'eaires $\P^{n-1} \subset \P^n$ qui contiennent $Q$.
On trouve ainsi une vari\'et\'e $Y\subset X \times \P^1$ munie d'un morphisme propre 
 $Y \to X$
et d'une fibration $Y \to \P^1$ dont les fibres au-dessus de
$k$-points   $s \in \P^1(k)$ sont des hypersurfaces cubiques $Y_{s} \subset \P^n_{k}$
sections hyperplanes de $X \subset  \P^{n}_{k}$ (l'hypoth\`ese que $X$ ne contient pas $Q$
garantit qu'aucun $Y_{s}$ n'est \'egal  \`a $\P^n_{k}$)
et dont la fibre g\'en\'erique est une hypersurface cubique $Y_{\eta} \subset \P^{n-1}_{K}$,
avec $K=k(\P^1)$. La droite $D$ d\'efinit une section de la  fibration $Y \to \P^1$, soit une courbe $M \subset Y$, dont
l'image se restreint en un
 $K$-point rationnel de $Y_{\eta}$. On dispose de la suite exacte
$$ \oplus_{s \in \P^1(k)} CH_{1}(Y_{s}) \to CH_{1}(Y) \to CH_{0}(Y_{\eta}) \to 0.$$
D'apr\`es le th\'eor\`eme \ref{generique}, la classe de $M$ dans $CH_{1}(Y)$ s'envoie sur un
g\'en\'erateur de  $CH_{0}(Y_{\eta})\simeq\Z$.
L'application $CH_{1}(Y) \to CH_{1}(X)$ est surjective. En effet le morphisme $Y \to X$
induit un isomorphisme au-dessus du compl\'ementaire du ferm\'e propre $X \cap Q \subset Q$, et au-dessus
de chaque point de $X \cap Q$, la fibre est une droite projective.
L'image de $M$ est la droite $D$ de $X$, chaque groupe $CH_{1}(Y_{s})$ est par hypoth\`ese
de r\'ecurrence engendr\'e par des droites de $Y_{s}$, dont les images dans $X$ sont
des droites de $X$.
\end{proof}


\begin{thebibliography}{99}

 

\bibitem{brucewall} J. W. Bruce et C.T.C. Wall,  On the classification of cubic surfaces,
J. London Math. Soc. (2) {\bf 19} (1979) 245--256.

 

\bibitem{CT} J.-L. Colliot-Th\'el\`ene, Hilbert's Theorem 90 for $K_{2}$, with application
to the Chow groups of rational surfaces, Invent. math. {\bf  71} (1983) 1--20.

 

\bibitem{coraytsfasman} D. Coray et M. Tsfasman,    
Arithmetic on singular del Pezzo surfaces,  Proc. Lond. Math. Soc. (3), {\bf 57}, 25--87 (1988).
  
                      

\bibitem{fulton} W. Fulton, Intersection Theory, Ergebnisse der Mathematik und ihrer Grenzgebiete, 3. Folge, Bd. {\bf 2}, Springer-Verlag (1984).

\bibitem{madore2008} D. Madore, \'Equivalence rationnelle sur les hypersurfaces cubiques de mauvaise r\'eduction,
    J. Number Theory {\bf 128} (2008), no. 4, 926--944. 
   


\bibitem{paranjape} K. Paranjape,  Cohomological and cycle-theoretic connectivity, Annals of Math.  {\bf 140} (1994) 641-660.

 

\bibitem{shen} Mingmin Shen, On relations among 1-cycles on cubic hypersurfaces, J. Alg. Geometry {\bf 23} (2014) 539-569.

\bibitem{shen2}  Mingmin Shen, Rationality, universal generation and the integral Hodge conjecture,
https://arxiv.org/abs/1602.07331v2

\bibitem{TZ} Zhiyu Tian et  Hong R. Zong, One-cycles on rationally connected varieties.  
Compositio. Math. {\bf 150}, no. 3 (2014) 396--408. 



\end{thebibliography}
\end{document}